\documentclass{amsart}

\usepackage{amssymb}
\usepackage{amsfonts}
\usepackage{paralist}
\usepackage{xspace}
\usepackage{euscript}
\usepackage{color}
\usepackage{epigraph}
\usepackage{algorithm}
\usepackage{tikz}
\usepackage[all]{xy}
\usepackage{cancel}
\usepackage[normalem]{ulem}
\usepackage[colorlinks,pagebackref=true,pdftex]{hyperref}
\usepackage[backrefs,abbrev,lite]{amsrefs} 

\makeatletter
\@addtoreset{equation}{section}
\def\theequation{\thesection.\@arabic \c@equation} 

\def\theenumi{\@roman\c@enumi}
\def\@citecolor{blue}
\def\@linkcolor{blue}
\def\@urlcolor{blue}

\makeatother

\newtheorem{lemma}[equation]{Lemma}
\newtheorem{prop}[equation]{Proposition}
\newtheorem{cor}[equation]{Corollary}

\newtheorem{claim*}{Claim}
\newtheorem{thm}[equation]{Theorem}
\newtheorem{question}[equation]{Question}

\theoremstyle{definition}
\newtheorem{remark}[equation]{Remark}
\newenvironment{rmk}[1][]{\begin{remark}[#1] \pushQED{\qed}}{\popQED \end{remark}}
\newtheorem{ex}[equation]{Example}

\newtheorem{defn}[equation]{Definition}

\newtheorem{notn}[equation]{Notation}

\def\<{\langle}
\def\>{\rangle}

\newcommand{\codim}{\operatorname{codim}}

\newcommand{\projdim}{\operatorname{pd}}
\newcommand{\coker}{\operatorname{coker}}

\newcommand{\initial}{\operatorname{in}}

\newcommand{\reg}{\operatorname{reg}}

\newcommand{\Tor}{\operatorname{Tor}}

\renewcommand{\to}{\longrightarrow}

\newcommand{\NN}{\mathbb{N}}

\newcommand{\QQ}{\mathbb{Q}}

\newcommand{\defi}[1]{{\bfseries\upshape #1}}

\newcommand{\pdim}{\operatorname{pdim}}

\newcommand{\excise}[1]{}

\title{Free Resolutions and Sparse Determinantal Ideals}
\author{Adam Boocher}
\address{Department of Mathematics, University of California,
        Berkeley, CA 94720-3840, USA}
\email{aboocher@math.berkeley.edu}
\urladdr{http://math.berkeley.edu/\~{}aboocher}
\thanks{
The author is partially supported by an NSF Graduate Fellowship.
}

\begin{document}
\maketitle
\begin{abstract} A sparse generic matrix is a matrix whose entries are distinct variables and zeros.  Such matrices were studied by Giusti and Merle who computed some invariants of their ideals of maximal minors.   In this paper we extend these results by computing a minimal free resolution for all such sparse determinantal ideals.  We do so by introducing a technique for pruning minimal free resolutions when a subset of the variables is set to zero.  Our technique correctly computes a minimal free resolution in two cases of interest: resolutions of monomial ideals, and ideals resolved by the Eagon-Northcott Complex.  As a consequence we can show that sparse determinantal ideals have a linear resolution over $\mathbb{Z}$, and that the projective dimension depends only on the number of columns of the matrix that are identically zero.  We show this resolution is a direct summand of an Eagon-Northcott complex.  Finally, we show that all such ideals have the property that regardless of the term order chosen, the Betti numbers of the ideal and its initial ideal are the same.  In particular the nonzero generators of these ideals form a universal Gr\"obner basis. \end{abstract}
\section{Introduction}

Let $S$ be a polynomial ring over $K$, where $K$ is any field or $\mathbb{Z}$.  By a \defi{sparse generic matrix}, we mean a $k\times n$ matrix $X' $ (with $k\leq n$) whose entries are distinct variables and zeros, and will denote by $I_k(X')$ its ideal of maximal minors, which we call a \defi{sparse determinantal ideal}.  For example, the two matrices below are both sparse generic matrices.  
\begin{figure}[h]$$X=\left(\begin{array}{cccc}
x_1 & x_2 & x_3 & x_4  \\
y_1 & y_2 & y_3 & y_4  \\
z_1 & z_2 & z_3 & z_4 
\end{array}\right) \ \ 
X' = \left(\begin{array}{ccccc}
0 & 0 & x_3 & 0 \\
0 & 0 & y_3 & y_4 \\
z_1 & z_2 & 0 & 0
\end{array}\right).
$$\label{3x4 example}\caption{A Generic Matrix and a Specialization} \end{figure}

Sparse generic matrices and determinantal ideals were studied by Giusti and Merle in \cite{MR708329} where they showed that the codimension, primeness, and Cohen-Macaulayness of $I_k(X')$ depend only on the perimeter of the largest subrectangle of zeros in $X'$.  
In this paper we continue the story by studying the homological invariants of these ideals and describe explicitly how to compute their minimal free resolution in terms of the arrangement of zeros.  It turns out that the minimal free resolution is always a direct summand of the Eagon-Northcott complex.  In addition, the projective dimension and regularity of such ideals are the same as in the generic case:

\begin{thm}\label{eagon- thm in intro} Let $X'$ be a $k\times n$ sparse generic matrix with no column identically zero, and $I=I_k(X')$ its ideal of maximal minors.  If $I\neq 0$ then $\reg S/I= k$ and $\pdim S/I=n-k+1$.   Further, if $X$ is a generic $k\times n$ matrix, then the resolution of $S/I_k(X')$ is a direct summand of the Eagon-Northcott complex associated to $X$ after specialization.  In particular,
$$\beta_{ij}(S/I_k(X'))\leq \beta_{ij}(S/I_k(X)), \ \ \mbox{for all } i,j.$$
\end{thm}

Sparse generic matrices can be thought of as generic matrices after setting some variables equal to zero.  For an arbitrary ideal, it is difficult to describe how the minimal free resolution changes after setting some linear forms equal to zero.  Indeed, the Betti numbers, projective dimension, and regularity can be wildly different before and after specialization.  However, in the case of determinantal ideals, which are resolved by Eagon-Northcott complex, there is a simple greedy algorithm that can be used to compute the minimal free resolution of any sparse determinantal ideal.  This is the basis for our proof of Theorem \ref{eagon- thm in intro}.  The following example illustrates our method: 

\begin{ex}\label{pruning ex}
Consider the matrices $X$ and $X'$ in Figure \ref{3x4 example}.  We begin with the Eagon-Northcott complex that resolves $S/I_3(X)$:
$$\xymatrix{ 0 \ar[r] & S^3 \ar[rrr]^{\begin{pmatrix} x_4 & y_4 & z_4\\x_3 & y_3 & z_3 \\ x_2 & y_2 & z_2  \\ x_1 & y_1 & z_1 \end{pmatrix}} & & & S^4 \ar[rrrrr]^{\begin{pmatrix} \Delta_{123} & -\Delta_{124} &\Delta_{134}&-\Delta_{234}\end{pmatrix}} &&&  && S }$$
where $\Delta_J$ denotes the minor indexed by the columns in $J$.  Now suppose we want to resolve $S/I_3(X')$.  Naively we might just set $x_1,x_2,x_4,y_1,y_2,z_3$ and $z_4$ equal to zero - i.e. tensor with $T=S/(x_1,x_2,x_4,y_1,y_2,z_3)$.  The result is: 
$$\xymatrix{0 \ar[r]& T^3 \ar[rrr]^{\begin{pmatrix} 0 & y_4 & 0\\x_3 & y_3 &  0\\ 0 & 0 & z_2  \\ 0 & 0 & z_1 \end{pmatrix}} & & & T^4 \ar[rrrrr]^{\begin{pmatrix} 0 & 0 & x_3y_4z_1 &-x_3y_4z_2\end{pmatrix}} &&& & & T }$$
Notice that the first two columns of the rightmost matrix are redundant, and hence, so are the first two rows of the leftmost matrix.  Deleting the corresponding summand of $T^4$ we obtain:
$$\xymatrix{0\ar[r] & T^3 \ar[rrr]^{\begin{pmatrix}  0 & 0 & z_2  \\ 0 & 0 & z_1 \end{pmatrix}} & & & T^2 \ar[rrrrr]^{\begin{pmatrix} x_3y_4z_1 &x_3y_4z_2\end{pmatrix}} &&& & & T }$$
And now similarly we prune the first matrix:
$$\xymatrix{ 0\ar[r]& T^1 \ar[rrr]^{\begin{pmatrix}   z_2  \\ z_1 \end{pmatrix}} & & & T^2 \ar[rrrrr]^{\begin{pmatrix} x_3y_4z_1 &-x_3y_4z_2\end{pmatrix}} &&& & & T }$$
In this case the resulting sequence of maps is a minimal free resolution of $T/I_3(X')$.   This exemplifies what we call the \defi{Pruning Technique.}
\end{ex}

We will define and study the pruning technique in Section \ref{section pruning method defn}.  Our main result on pruning is the following:

\begin{thm}\label{pruning works intro}
Suppose that $I\subset S$ is an ideal in a polynomial ring and $Z$ is a subset of the variables.  If $T=S/(Z)$ then the pruning technique computes a minimal free resolution of $S/I\otimes T$ as a $T$-module in the following two cases:
\begin{itemize}
\item $I$ is a monomial ideal.
\item $I$ is a determinantal ideal resolved by the Eagon-Nortcott Complex
\end{itemize}
\end{thm}
In Section 2 we also discuss a homological interpretation of pruning.  One feature of this interpretation is that it can be used (see Corollaries \ref{mon I:x/I res} and \ref{I:x/I has a linear resolution}) to describe the shape of the Betti table of $\Tor_1(S/I,S/(x))$ where $x$ is a variable and $I$ is either a monomial ideal or a sparse determinantal ideal.   

Our proof of Theorem \ref{pruning works intro} proceeds in two cases.   For monomial resolutions, we study an $\mathbb{N}^n$ grading.  For determinantal ideals, we use the result of Sturmfels, Zelevinsky, and Bernstein \cite{MR1212627,MR1229427} that shows that the maximal minors of a generic matrix are a universal Gr\"obner basis for the ideal $I$ that they generate.  Since setting variables equal to zero is \emph{almost} like taking them last in  a term order, it is natural to study the free resolution of initial ideals of $I_k(X)$ when $X$ is generic.    For example, the aforementioned Gr\"obner basis result says that for any term order ``$<$'', 
$$\beta_1(S/I)=\beta_1(S/\initial_< I).$$  We extend this to show that in fact the maximal minors are a universal Gr\"obner resolution in the following sense:

\begin{thm}\label{thm linear res intro}
Let $X$ be a (sparse) generic matrix and let $I$ denote its ideal of maximal minors.  Then for any term order $<$, we have
$$\beta_{ij} (S/I) =  \beta_{ij} (S/\initial_< I) \ \mbox{ for all } i,j.$$
In particular, every initial ideal of $I$ has a linear resolution.
\end{thm}

We note that the analagous result does not hold for lower order minors.  In fact, even the $2\times 2$ minors of a $3\times 3$ matrix are not a universal Gr\"obner basis. \cite{MR1212627}

Theorem \ref{thm linear res intro} provides a new class of squarefree Cohen-Macaulay monomial ideals generated in degree $k$ that have a linear resolution.   
Combining the techniques of pruning and taking initial ideals, we can obtain a class of squarefree monomial ideals with linear resolutions that sit inside of the Eagon-Northcott complex.  Finally, although the proofs rely on the Gr\"obner basis property, the pruning algorithm itself is algebra free - it only involves an eraser.




\section{The Pruning Technique}\label{section pruning method defn}
In this section we define and study the pruning technique.  Throughout, $S$ will denote a polynomial ring over $K$, where $K$ is any field or $\mathbb{Z}$.  The variable names may change for convenience, but should always be clear from the context.   By $Z\subset S$ we will always mean a subset of the variables or as an an abuse of notation, the ideal that they generate in $S$.   We set $T:=S/Z$. 

The pruning technique is a way of approximating a $T$-resolution of $M\otimes T$ starting from an $S$-resolution of $M$.   To do so, we essentially tensor the given resolution with $T$ and erase any obvious excess.  The definition here - which makes precise the method outlined in Example \ref{pruning ex} - requires a choice of basis, but as we will discuss later, this is mostly for convenience.



\begin{defn}
Let $C_\bullet$ be a complex of free $S$-modules with choice of bases (so we have a matrix for each map)
$$\xymatrix{F_t \ar[r]^{A_t} & F_{t-1} \ar[r] & \cdots \ar[r] & F_1 \ar[r]^{A_1} & F_0}.$$
Let $Z$ be a subset of the variables.  We define the \defi{pruning of $C_\bullet$ with respect to $Z$} to be the complex of $T:=S/Z$-modules obtained from $C_\bullet$ by the following algorithm:  

{\tt Let $i=1$

For $i\leq t$ do:

In the matrix $A_i$, set all variables in $Z$ equal to zero.   Set $A_i$ \\ equal to this new matrix, and set $U$ equal to the set indexing which \\ columns of $A_i$ are identically zero.  

Replace, $\{A_{i+1}, F_i, A_i\}$ with news maps, and modules obtained by simply \\ deleting the rows, basis elements, and columns, respectively, \\corresponding to $U$.

   Let $i=i+1$.  }

The resulting sequence of maps with bases is naturally a sequence of $T$-modules, which we will denote $P(C_\bullet, Z)$.  
\end{defn}

\begin{prop}
If $C_\bullet$ is a complex, then so is $P(C_\bullet ,Z)$.  In addition, if the entries of the matrices of $C_\bullet$ are in the homogeneous maximal ideal, then the same is true for those of $P(C_\bullet,Z)$.  
\end{prop}
\begin{proof}
It is clear that if $A_i\cdot A_{i-1} = 0$ then the same is true once we set variables in $Z$ equal to zero.   Further, any column that is identically $0$ in $A_{i-1}$ essentially makes the corresponding row in $A_i$ irrelevant for the product to be zero.   Indeed the non-identically-zero columns of $A_{i-1}$ must now necessarily pair to zero with the corresponding rows of $A_i$.  This is exactly what the pruning process does.  Finally, since pruning only erases entries, the second claim of the Proposition is clear.
\end{proof}

In some cases, the pruning technique preserves exactness:

\begin{thm}\label{monomial case}
Let $I$ be a monomial ideal in a polynomial ring $S$ with $n$ variables, and let $C_\bullet$ be a minimal free resolution of $S/I$ with $\NN^n$ homogeneous bases.  If $Z$ is an ideal generated by a subset of the variables then $P(C_\bullet,Z)$ is a minimal free resolution of $S/I\otimes S/Z$ as an $S/Z$ module. 
\end{thm}

The proof follows from a careful study of the $\NN^n$ grading.  We will use a similar technique below to study the case of the Eagon-Northcott complex.  
\begin{proof}
We may assume that $Z=(x_1,\ldots,x_r)$.  By the grading of $C_\bullet$, the maps will be of the form
\vspace{.1in}
\[
 \begin{matrix} \displaystyle{\bigoplus_{\substack{\mbox{\tiny{not all}} \\ b_{i1},\ldots,b_{ir} = 0}} S(-b_{i1},\ldots -b_{ir},\ldots,-b_{in})} \\ \bigoplus \\ \bigoplus  S(0,\ldots,0,-c_{i\ r+1},\ldots,-c_{in}) \end{matrix} \stackrel{M_i}{\xrightarrow{\hspace*{1cm}}} \begin{matrix} 
\displaystyle{ \bigoplus_{\substack{\mbox{\tiny{not all}} \\ a_{j1},\ldots,a_{jr} =  0 }} S(-a_{j1},-a_{j2},\ldots,-a_{jn})} \\ \bigoplus \\ \bigoplus S(0,\ldots, 0,-d_{j\ r+1},\ldots,-d_{jn})
  \end{matrix}
\]

where the matrix $M_i$ has the form
$$\left(\begin{array}{c|c}
A_i & 0 \\
\hline 
C_i & D_i
\end{array}\right).$$
By the grading, it is clear that every nonzero entry in the submatrix $C_i$ is divisible by some $x_i\in Z$.  In this notation, the beginning of the resolution of $S/I$ is:
$$\xymatrix{F_2 \ar[rr]^{\left(\begin{array}{c|c}
A_2 & 0 \\
\hline 
C_2 & D_2
\end{array}\right)} & & F_1\ar[rr]^{\left(\begin{array}{c|c} C_1 & D_1\end{array}\right)} & & S}$$
where the first matrix is a row matrix consisting of the generators of $I$.  
Thus the pruning algorithm, will commence by deleting the columns in $C_1$, the rows of $A_2$, obtaining 

$$\xymatrix{F_2 \ar[rr]^{\left(\begin{array}{c|c}
C_2 & D_2
\end{array}\right)} & & F_1'\ar[rr]^{\left(\begin{array}{c}D_1\end{array}\right)} & & S}.$$
Now, inductively we can see that the pruning algorithm will successively prune each matrix $M_i$ down to the matrix $D_i$.  Hence $P(C_\bullet,Z)$ is the complex of $T$ of modules whose $i$th map is given by $D_i$.    

To see that $P(C_\bullet, Z)$ is a resolution, notice that any element $v=(v_1,\ldots,v_k)$ in the kernel of $D_i$ trivially extends to the element $w=(0,\ldots,0,v_1,\ldots,v_k)$ which is in the kernel of $M_i$.   By the exactness of the original complex, we deduce that $w$ is in the image of $M_{i+1}$, say $w=M_{i+1}(u)$.  Finally, since every entry of $C_{i+1}$ is zero mod $Z$, we have the following equality over $S/Z$: 
$$v=\pi(w)=\pi(M_{i+1}(u)) = (C_{i+1}|D_{i+1})(u)=D_{i+1}(\overline{u})$$
where $\pi$ is the obvious projection sending $w$ to $v$ and $\overline{u}$ consists of the last entries of $u$.  Hence mod $Z$, $v$ is in the image of $D$.  
\end{proof}

The pruning technique does not preserve exactness in general, as the following example shows:

\begin{ex}\label{buchsbaum rim counterexample}
Consider the Buchsbaum-Rim resolution of the generic 2 by 3 matrix $M$:
$$\xymatrix{0\ar[r] &  S^1 \ar[rr]^{\begin{pmatrix}  \Delta_{23}  \\ -\Delta_{13}  \\ \Delta_{12}  \end{pmatrix}} & &  S^3 \ar[rrr]^{\begin{pmatrix} x&y&z \\ a& b&c\end{pmatrix}} & & & S^2}.$$
This is a minimal free resolution of $\coker M.$  Here $\Delta_{ij}$ denotes the $ij$ minor of the presentation matrix.   Pruning by setting $x$  and $y$ to zero yields 
$$\xymatrix{0\ar[r] &  T^1 \ar[rr]^{\begin{pmatrix}  bz  \\ -az  \\ 0  \end{pmatrix}} & &  T^3 \ar[rrr]^{\begin{pmatrix} 0 & 0& z \\ a& b&c\end{pmatrix}} & & & T^2}.$$
This is not exact since the kernel of the righthand map contains the element $(b, -a, 0)^T$, which is not in the image of the first.
\end{ex}

We note that the pruning process has only been defined for complexes with a choice of bases.   We have chosen this definition because it is all we need for the main results in this paper, and we feel that it highlights the important aspects of monomial resolutions, and the Eagon-Northcott complex. However, we could easily modify our definition to allow row and column operations over $K$.  In fact, pruning can be defined without referring to matrices at all, simply by tensoring the given resolution with $T$ and then taking successive quotients by the free module of degree zero syzygies at each stage.  A further generalization might be to also include saturating by dividing through by common factors, which would remedy the problem with Example \ref{buchsbaum rim counterexample}.  We plan to study this generalization in the future. 

Another interpretation of pruning is as follows: If $F_\bullet \to M$ is a minimal free resolution and $x$ is a variable, then a general pruning technique should ``work'' exactly when the minimal free resolution of $M\otimes S/(x)$, is a direct summand of $F_\bullet \otimes S/(x)$.  The following general result gives a necessary and sufficient condition for this to occur.

\begin{prop}\label{gen characterization of pruning}
Let $F_\bullet$ be a minimal free resolution of a graded $S$-module $M$ and let $x \in S$ be any homogeneous polynomial.  By $F_\bullet '$ we will denote the complex of $S/(x)$-modules obtained by tensoring $F_\bullet$ with $S/(x)$.  If $H$ denotes $H^{S/(x)}_1(F_\bullet ')$, then the following are equivalent:
\begin{enumerate}
\item The minimal free resolution of $M':=M\otimes S/(x)$ is a direct summand of $F_\bullet '$. 
\item There is a split inclusion of the minimal free resolution of $H$ as an $S/(x)$-module into $F_\bullet'[1]$.  
\end{enumerate}
\end{prop} 
\begin{proof}
We being by noting that since $\Tor^S_i(M,S/(x))=0$ for $i>1$ we have $H_j(F_\bullet ')=0$ for all $j>1$.

$(i)\implies (ii)$:  Let $G_\bullet$ be a minimal free resolution of $M'$. Then $(i)$ says that there are projection maps $\pi$ such that the following diagram commutes: 
$$\xymatrix{  \cdots \ar[r]& G_n \ar[r] & \cdots \ar[r] &  G_1 \ar[r] & G_0 \\
\cdots \ar[r] & F_n'  \ar[r]\ar[u]_{\pi} & \cdots  \ar[r] & F_1'  \ar[r]\ar[u]_{\pi} & F_0' \ar[u]_{\pi} 
}.$$
Letting $K_\bullet$ denote $(\ker \pi)_\bullet$, we see that $K_\bullet$ split injects into $F_\bullet'$.  To see that $K_\bullet$ is a resolution of $H$, notice that the long exact sequence of homology implies that
$$\cdots \to H_{i+1}(G_\bullet) \to H_{i}(K_\bullet) \to H_{i}(F_\bullet ')\to H_{i}(G_\bullet)\to \cdots$$
is exact.  Since $H_j(F_\bullet ')=0$ for all $j>1$, and $G_\bullet$ is exact, we conclude that $H_j(K_\bullet)=0$ for $j\geq 2$.  Finally, we obtain the exact sequence:
$$0 \to H_{1}(K_\bullet) \to H_{1}(F_\bullet ')\to 0 \to H_{0}(K_\bullet) \to M' \stackrel{=}{\to} M'\to 0$$ and we see that $H_1(K_\bullet) \cong H$, so that $K_\bullet[-1]$ is a minimal free resolution of $H$, and hence $K_\bullet$ split-injects into $F_\bullet'[1]$.  

$(ii) \implies (i)$:  Suppose that we have a minimal free resolution $K_\bullet \to H$ which split injects into $F_\bullet '[1]$.  Then we have the following commutative diagram: 
$$\xymatrix{  \cdots \ar[r]& F_n' \ar[r] & \cdots \ar[r] &  F_1' \ar[r] & F_0' \\
\cdots \ar[r] & K_{n-1}  \ar[r]\ar[u]_{\phi} & \cdots  \ar[r] & K_0  \ar[r]\ar[u]_{\phi} & 0 \ar[u]_{\phi}
}.$$
Taking cokernels of each map, and applying the long exact sequence of homology as in the first part of the proof, we see that $(\coker {\phi})_\bullet$ is a minimal free resolution of $M'$.  
\end{proof}

\begin{rmk}\label{rmk about mapping cones}
Notice that in general, if $K_\bullet \to H$ is a resolution, then there is always a (non-canonical) map of complexes:  $\phi: K_\bullet \to F_\bullet'[1]$.  The mapping cone of $\phi$ will be a (typically non-minimal) free resolution of $M'$.  In cases where pruning works, $\phi$ can be taken to be an inclusion.  
\end{rmk}

\begin{cor}\label{mon I:x/I res}
If $I$ is a monomial ideal, and $x$ is a variable, then 
$$ \beta_{ij} \left(\frac{I:x}{I}\right)  \leq \beta_{ij}(I) \ \mbox{for all $i,j$}$$
\end{cor}
\begin{proof}
Let $F_\bullet \to S/I$ be a minimal free resolution.  By Theorem \ref{monomial case}, the minimal free resolution of $S/I \otimes S/(x)$ is a direct summand of $F_\bullet' =F_\bullet \otimes S/(x)$.  Hence by Proposition \ref{gen characterization of pruning}, the resolution of $H_1(F_\bullet') \cong (I:x)/I$ is a direct summand of $F_\bullet'[1]$.  In particular, the degrees and ranks of the free modules appearing in a minimal free resolution of $(I:x)/I$ can be no larger than those appearing in $F_\bullet'[1]$.  Since $F[1]$ is a minimal free resolution of $I$ we see the desired inequality. 
\end{proof}



\section{Initial Ideals of $I_k(X)$}\label{initial ideal section}
In order to prove Theorems \ref{eagon- thm in intro} and \ref{pruning works intro}, it is useful to study the various initial ideals of $I_k(X)$ when $X$ is a generic $k\times n$ matrix.  By term order, we will always mean a monomial term order $<$, so that the initial ideal will be monomial. 

In general, when passing to an initial ideal, we expect homological invariants to change.  Indeed, since passing to the initial ideal is a flat deformation, we have
$$ \beta_{ij}(S/\initial_< I)\geq \beta_{ij} (S/I) \mbox{ for all } i,j$$
and typically these inequalities are strict.  (For a great exposition, see \cite{MR2724673}). For instance, the first Betti numbers are equal if and only if the ideal is minimally generated by a Gr\"obner basis with respect to the term order.  In this vein, Sturmfels, Zelevinsky, and Bernstein have shown in \cite{MR1212627,MR1229427} that the maximal minors form a universal Gro\"bner basis for $I:=I_k(X)$.  This proves, for instance, that $\beta_1(S/\initial_< I)=\beta_1 (S/I) =  {n \choose k}$ for any term order.  In this section we prove 

\begin{thm}\label{all initial ideals have linear resolutions}
If $I:=I_m(X)$ is the ideal of maximal minors of a generic matrix $X$ and $<$ is any term order, then 
$$ \beta_{ij}(S/\initial_< I)= \beta_{ij} (S/I) \ \ \mbox{ for all } i,j.$$
In particular, every initial ideal is a Cohen-Macaulay, squarefree monomial ideal with a linear free resolution.  Further, the resolution can be obtained from the Eagon-Northcott complex by taking appropriate lead terms of each syzygy.
\end{thm}

For certain orders, analyzing the initial ideal explicitly is manageable.  For example, diagonal term orders were viewed in the context of basic double links in \cite{Gorla:2010fk} where they proved such initial ideals are Cohen-Macaulay.  In general, however, not all term orders have ``nice'' descriptions.  Instead we use the following fact:
 
 \begin{lemma}\label{primary decomp label}[Sturmfels-Zelevinsky~\cite{MR1212627}]
 For any monomial term order $<$, the initial ideal $\initial_< I$ is squarefree and has a primary decomposition of the form 
 $$\initial_< I = \bigcap_\alpha I_\alpha$$
where $\alpha$ ranges over all subsets $\{j_1,j_2,\ldots,j_c\}$ of $\{1,\ldots,n\}$ with $c=n-k+1$, and $I_\alpha = (x_{i_1j_1},\ldots,x_{i_cj_c})$ for some indices $i_1,\ldots, i_c$ which depend on the term order and $\alpha$.
 \end{lemma}

\begin{remark}\cite{MR1212627} gives an explicit description of the components $I_\alpha$ in terms of the monomial order $<$, but we will not need that much detail in what follows.  \end{remark}

\begin{proof}[Proof of Theorem \ref{all initial ideals have linear resolutions}]  Let $<$ be any term order, and write $\initial I = \initial_< I$.  We will show that
$$\{x_{11}-x_{21},\ldots,x_{11}-x_{k1}\}\cup \{x_{12}-x_{22},\ldots, x_{12}-x_{k2}\}\cup\cdots$$ $$\cdots \cup\{x_{1n}-x_{2n},\ldots, x_{1n}-x_{kn}\}$$
is a regular sequence on $S/\initial I$.   Indeed, once this is shown, we know that the Betti numbers of $\initial I$ are the same as those of the ideal obtained by substituting the relations induced by the regular sequence above.  These are precisely the substitutions $x_{ij}=x_{1j}$ for all $i,j$.  Since $\initial I$ is the ideal generated by the leading term of each minor, these substitutions deform $\initial I$ into  the ideal $J$ consisting of all squarefree degree $k$ monomials in $K[x_{11},\ldots,x_{1n}]$.  The resolution of this ideal is well known.  
In particular, its Betti numbers are equal to those in the Eagon-Northcott complex, and $\beta_{ij}(S/\initial I)= \beta_{ij}(S/J) = \beta_{ij}(S/I)$ as required. 

To prove that the sequence defined above is a regular sequence, we successively modify the primary decomposition described in Lemma \ref{primary decomp label} after each substitution.  Since in the end, we will only compute with the ideal formed by substituting $x_{ij} =x_{1j}$, we study these \emph{substitution ideals}.  

Set $K=\initial I$ and suppose $K = \bigcap P_i$ as in the Lemma.  Since we will inductively apply the following argument, we first highlight the following properties that we will use about $K$:
\begin{itemize}
\item $K$ has no minimal generators that contain a product of two elements from the same column of $X$.  
\item The ideals $P_i = (x_{i_1j_1},\ldots,x_{i_cj_c})$ are generated by variables in different columns of $X$. 
\end{itemize}
Let $x_{ij}$ be any variable with $i\neq 1$.  For the ease of notation, we will write $sub$ to denote the substitution $x_{ij}\to x_{1j}$.  We claim that the following two monomial ideals are equal: 

$$(K)_{sub} = \bigcap (P_i)_{sub}.$$

Indeed, since substitution is just a ring map, $K\subset \cap P_i$ implies that $K_{sub}\subset \cap (P_i)_{sub}$.  


Conversely, suppose that $f$ is a minimal generator of $\cap (P_i)_{sub}$.  Notice that $f$ does not involve $x_{ij}$.  We have two cases:

Case 1:  $x_{1j}$ does not divide $f$.   In this case, the membership of $f$ in $(P_i)_{sub}$ guarantees membership in $(P_i)$ since the factors of $f$ relevant to ideal membership do not change under our substitution.

Case 2:  $x_{1j}$ divides $f$, say $f=x_{1j}g$.  Consider the element $h=x_{ij}f.$  Since $h$ is divisible by both $x_{ij}$ and $x_{1j}$, and since $f$ is in $\cap(P_i)_{sub}$, we know $h$ is in fact in each ideal $P_i$.  Thus $h=x_{ij}f =x_{ij} x_{1j} g\in K$.  But since $K$ has no minimal generators divisible by $x_{ij} x_{1j}$ we know that either $x_{ij}g$ or $x_{1j}g$ must be in $K$.  Under the substitution, both of these elements will be sent to $f$, so that $f\in K_{sub}.$

Notice that if we next replace $K$ and $P_i$ with $(K)_{sub}$ and $(P_i)_{sub}$, then $K$ and $P_i$ still satisfy the bulleted properties above. Therefore, we may inductively apply our argument to the next substitution $x_{ij} \to x_{1j}$ to complete the proof.  
\end{proof}

\begin{remark}
It is a very rare property for an ideal be minimally generated by a universal Gr\"obner basis, and it is an even rarer property for $\beta_{ij} (S/I) = \beta_{ij} (S/\initial_< I)$ for all $i,j \geq 0$, for every term order.   Indeed, there are ideals that are minimally generated by a Gr\"obner basis, but whose initial ideals still have strictly larger Betti numbers than those of the ideal itself.  For example, $\{ab,bc,cd,de,ae+ac\}$ is a universal Gr\"obner basis, but its two initial ideals have distinct Betti tables. 
\end{remark}

\begin{question}
What conditions are necessary and sufficient to guarantee $\beta_{ij} (S/I) = \beta_{ij} (S/\initial_< I)$ for all $i,j \geq 0$, for every term order?
\end{question}

Having shown the Betti numbers of $S/I$ and $S/\initial_< I$ are equal, a natural question is how to obtain a minimal free resolution for $S/\initial I$.  We next show that this can easily be obtained from the Eagon-Northcott complex.  

Since our pruning technique is defined only for complexes where the maps are represented by matrices, we need to specify what we mean by ``Eagon-Northcott complex''.   By this, we will always mean the complex whose first map consists of the minors $\Delta_J$ and whose later maps are of the form 
$$D_a(S^k) \otimes \wedge^{a+k} (S^n) \to D_{a-1} (S^{k}) \otimes \wedge^{a+k-1} (S^n)$$
where $D_i$ is the divided power algebra and the matrices are chosen with respect to the natural basis $e_1^{(n_1)}\cdots e_k^{(n_k)}\otimes f_{j_1}\wedge \cdots \wedge f_{j_\ell}$, where $e_1,\ldots,e_k$ and $f_1,\ldots,f_n$ are bases for the rows and columns of $X$.   

\begin{rmk}\label{no rows or columns have the same}
Notice that with this choice of basis, the first matrix in the complex consists of the minors $\Delta_J$, and all syzygy matrices are essentially multiplication tables between the rows and columns.  For this reason we notice that each entry is simply a variable $\pm x_{ij}$ and that no variable appears twice in the same row or column.
\end{rmk}

Now let $w$ be any set of weights on the variables $x_{ij}$.  Then since we can always choose a monomial order $<_w$ which refines that of $w$, we have
$$\beta_{ij} (S/I) \leq \beta_{ij} (S/\initial_w I)\leq \beta_{ij} (S/\initial_{<_w} I).$$
By Theorem \ref{all initial ideals have linear resolutions}, we have equality.

For a weight $w$, we can homogenize any $f\in S$ by taking the leading term to be the one of highest weight, and multiplying smaller order terms by appropriate powers of a parameter $t$.  We denote the homogenization $f^h$ and will write $I^h$ for the ideal 
$$I^h = \{f^h\ | \ f \in I\}\subset S[t].$$
Similarly, we can homogenize any map between free $S$-modules.  

\begin{ex}
If we consider the Eagon-Northcott complex on the matrix with weights
$$X = \begin{pmatrix} x & y & z \\ a & b & c \end{pmatrix}, \ \ w = \begin{pmatrix} 1& 1 & 2\\ 2 & 2 & 2  \end{pmatrix}$$
we could homogenize the maps to obtain
$$\xymatrix{ 0 \ar[r] & {\begin{array}{c} S[t](-5) \\ \oplus \\ S[t](-6) \end{array}} 
\ar[rrr]^{\begin{pmatrix} z & ct  \\ -y & -b  \\ x & a\end{pmatrix}}
 && &  {\begin{array}{c} S[t](-3) \\ \oplus \\ S[t](-4) \\ \oplus \\ S[t](-4) \end{array}} \ar[rrr]^{\begin{pmatrix} \Delta_{12}^h, \Delta_{13}^h,\Delta_{23}^h \end{pmatrix}} &&& S[t] }$$
 where $\Delta_{12}^h = xb-ay, \ \Delta_{13}^h = xct - az, \ \Delta_{23}^h = cyt - bz.$
\end{ex}
In this example, the above is a minimal free resolution of $I^h$.  This is always true, which we prove now. 
\begin{prop}\label{prop homogenizing ok}
Let $w$ be an integral weight order on the variables and let $E_\bullet$ denote the Eagon-Northcott complex.  Then $E_\bullet^h$ is a minimal free resolution of $S[t]/I^h$. 
\end{prop}
\begin{proof}
We notice that $I^h=(\Delta_J^h)$ since the $\Delta_J$ form a universal Gr\"obner basis, so we just need to show that $E_\bullet^h$ is exact.   To show this, it suffices to show that $E_\bullet^h$ is exact after tensoring with $S[t]/(t)$ - in other words, after erasing each entry divisible by $t$.  By Remark \ref{no rows or columns have the same} the surviving columns of each matrix will be linearly independent over $K$.  But since 
$$\beta_{ij}(S/I) = \beta_{ij}(S/\initial_w I)= \beta_{ij}(S/I^h) \mbox{ for all } i,j.$$
we see that these columns in fact span the full space of syzygies.
\end{proof}

\begin{cor}
To obtain the minimal free resolution of $S/\initial_< I$ simply set $t=0$ in the resolution $E_\bullet^h$ defined above.  
\end{cor}


\section{Minimal Free Resolution of Determinantal Ideals}
In this section we compute the minimal free resolution of the ideal $I_k(X')$ where $X'$ is a sparse generic matrix.  This section was inspired by the work of Giusti and Merle in \cite{MR708329}.   Throughout this section, $X$ and $X'$ will denote generic and sparse generic matrices respectively. 

Since a matrix with a column identically equal to zero is essentially a $k\times (n-1)$ matrix, we will assume $X'$ has no column identically zero.  We also assume that $I_k(X')$ is not the zero ideal.  This is equivalent to the fact that there is no rectangle of zeros in $X'$ whose perimeter is greater than $2n+1$.  (See \cite{MR708329})

\begin{thm}\label{EN proof}
Let $X=(x_{ij})$ be a generic $k\times n$ matrix and $Z$ be a subset of the variables.  Let $X'$ be the sparse generic matrix with variables in $Z$ set to zero.   If $E_\bullet$ is the Eagon-Northcott Complex with standard bases that resolves $S/(I_k(X))$, then the result of pruning - $P(E_\bullet,Z)$ is a minimal free resolution of $S/I_k(X')$ as an $S/Z$ module.  
\end{thm}

\begin{proof}
Let $I=I_k(X)$.  
To simplify notation, we will use $z_{ij}$ to denote the variables in $Z$, and use $x_{ij}$ to denote the other variables.   Assign a grading on $S$ by assigning weights
$$w(z_{ij}) = 1, \ \ w(x_{ij}) = 2.$$
Under this grading, the ideal $I$ is no longer homogenous.  

By Proposition \ref{prop homogenizing ok} $E_\bullet^h$ is a resolution of $S[t]/I^h$.  
In particular, 
$$I^h = (\Delta_J^h)$$
where $J$ runs over all the $k\times k$ minors.

 Further, there is a dichotomy
 $$w(\Delta_J^h) = 2k \iff  \Delta_J \neq 0 \mod Z,$$
 $$w(\Delta_J^h) < 2k \iff \Delta_J = 0 \mod Z.$$

By virtue of the simplicity of the maps in the Eagon Northcott complex, every matrix after the first contains entries that are simply variables of $S$.  Hence, with respect to our grading every element in these matrices is either of degree one or two before homogenization.   After homogenizing we can split our resolution into pieces:  One corresponding to the strand that resolves the ``surviving'' minors of weight $2k$, and the other consisting of everything else.  Explicitly, the $i$th map of $E_\bullet^h$ will look like: 

$$\xymatrix {\mbox{$\begin{array}{c} \bigoplus_{a_j<2k+2i+2} S[t](-a_j) \\ \bigoplus \\ \bigoplus S[t](-2k-2i-2) \end{array}$} \ar[rr]^{M_i} && \mbox{$\begin{matrix} \bigoplus_{b_j<2k+2i} S[t](-b_j) \\ \bigoplus \\ \bigoplus S[t](-2k-2i)
\end{matrix}$} }
$$
where the matrix $M_i$ has the form
$$\left(\begin{array}{c|c}
A_i & T_i \\
\hline 
C_i & D_i
\end{array}\right).$$
From the grading alone we can deduce three things:
\begin{itemize}
\item The nonzero entries of $T_i$ are divisible by $t$ since the degree shift is more than two. 
\item The nonzero entries of $C_i$ have degree at most one. (i.e. they are $z_{ij}$)
\item The nonzero entries of $D_i$ have degree two (i.e. they are $z_{ij}t$ or $x_{ij}$.)
\end{itemize}
Note that this implies that if we take the matrix $D_i$ modulo $Z$ or modulo $t$ we get the same result.   Denote this matrix $F_i$:
$$F_i := D_i \textrm{ mod } t = D_i \textrm{ mod } Z.$$
Therefore when we set $t$ equal to zero in $E_\bullet^h$, we obtain a complex $E'_\bullet$ where all matrices take the following form
$$\left(\begin{array}{c|c}
A_i & 0 \\
\hline 
C_i & F_i
\end{array}\right).$$
This is analogous to the decomposition we had in the monomial case.   By the same argument in the proof of Theorem \ref{monomial case} we conclude that modulo the variables in $Z$, the complex $F_\bullet$ is equal to $P(E_\bullet,Z)$ and is a minimal free resolution of $S/I\otimes S/(Z)\cong S/I_k(X')$.
\end{proof}

\begin{cor} \label{proj dim and reg are constant} If $X$ and $X'$ are as above, then \begin{itemize} 
\vspace{.1in}
\item $S/I_k(X')$ has regularity $k$ 
\vspace{.1in}
\item $\beta_{ij}(S/I_k(X')) \leq \beta_{ij}(S/I_k(X))$ for all $i,j$. 
\vspace{.1in}
\item $S/I_k(X')$ has projective dimension $n-k+1$.
\end{itemize}\end{cor}

\begin{proof}
Let $I'=I_k(X')$.  By Theorem \ref{EN proof}, the minimal free resolution of $S/I'$ is given by pruning the Eagon-Northcott complex, and as such, the degrees of syzygies do not change.  Hence the regularity is equal to $k$, the generating degree of the ideal, which proves the first statement.

Notice that each time we add a zero to our matrix, we can compute a minimal free resolution by pruning, and as such the Betti numbers can only possibly decrease. This shows the second statement.  

We compute the projective by using induction on $k$ and $n$.  Since the only $1\times n$ matrices with no columns identically equal to zero are generic matrices, the base case is trivial.  Similarly, $k\times k$ matrices give rise to a principal ideal of minors, which have projective dimension 1. 

Since $I'$ is nonzero, we can assume without a loss of generality that $D=\Delta_{1\cdots k}\neq 0$, and that a nonzero term of $D$ is a multiple of $x_{k1}$.  Notice that by pruning, the projective dimension can only decrease by adding more zeros, so it is sufficient to compute the projective dimension in the case when the first column has $k-1$ zeros.  Thus we may assume $X'$ has the form

$$X'=\begin{pmatrix} 0 & \star & \ldots & \star \\ \vdots & \vdots & \cdots & \vdots \\ 0 & \star & \cdots &\star \\ x_{k1} & \star &\cdots &\star  \end{pmatrix}=\begin{pmatrix} 0 & && \\   \vdots & & M' & \\ 0 &&  &  \\ \hline x_{k1} & \star&\cdots & \star \end{pmatrix}$$

 Let $Y$ denote the matrix of the rightmost $n-1$ columns of $X'$.  Then 
$$I':x_{k1} = I_{k-1}(M')\ \ \mbox{and} \ \ (I',x_{k1}) = (x_{k1}) + I_k(Y).$$
$M'$ is a sparse  generic matrix and since the minor $\Delta_{2\cdots n}$ (indices refer to those of $X'$) of $M'$ is nonzero by assumption, $I_{k-1}(M')$ is nonzero.  We have two cases:
\begin{itemize}\item Case 1: Suppose that some column $j$ of $M'$ is identically zero.  Then since $D\neq 0$ we know that $j>n$, and since $X'$ had no column identically zero, the $kj$ entry of $X'$ must be nonzero.  Hence $\Delta_{\{2\cdots n\}\cup \{j\}}\neq 0$, so that $I_k(Y)$ is nonzero.   In this case, by induction, $\pdim S/I_k(Y) = n-k$. 
\medskip \item Case 2: If no column of $M'$ is identically zero, then by induction, 
$$\pdim S/(I':x_{k1})=\pdim S/I_{k-1}(M)= n-k+1, \ \ \pdim S/I_k(Y) \leq n-k$$ 
the last inequality is strict if and only if $I_k(Y)$ is the zero ideal. 
\end{itemize}
In either case, we have 
$$\max \ (\pdim S/I_k(Y)+1, \pdim S/I_{k-1}(M')) = n-k+1.$$

Since the resolution of $S/(I',x_{k1})$ can be obtained by tensoring the resolution of $S/I_k(Y)$ with the Koszul complex on $x_{k1}$ we see that 
$$\pdim S/(I', x_{k1}) = \pdim S/I_k(Y) + 1$$ and that the minimal free resolution of $S/(I',x_{k1})$ is linear after the first map.  Applying the Horseshoe Lemma to the exact sequence
$$\xymatrix{0\ar[r] &S/(I':x_{k1})(-1)\ar[r]& S/I' \ar[r]& S/(I',x_{k1})\ar[r]& 0},$$
we see that a free resolution of $S/I'$ can be computed as the direct sum of the minimal free resolutions of $S/(I':x_{k1})$ and $S/(I',x_{k1}).$  Finally, since $S/(I':x_{k1})\cong S/I_{k-1}(M')$ has a linear resolution by Theorem \ref{EN proof}, this implies that except for the extra generator in homological degree $0$, the direct sum of the resolutions of the outside two modules is in fact a \emph{minimal} free resolution of $S/I'$.  Hence:

\begin{eqnarray*}\pdim S/I' &=& \max \ (\pdim S/(I',x_{k1}),\pdim S/(I':x_{k1}))\\ 
&=& \max \ (\pdim S/I_k(Y)+1, \pdim S/I_{k-1}(M'))\\
&=& n-k+1.\end{eqnarray*}
\end{proof}

We close this section by proving a result analogous to Corollary \ref{mon I:x/I res}.

\begin{cor}\label{I:x/I has a linear resolution}
Let $X'$ be a sparse $k\times n$ generic matrix and $I=I_k(X')$.  If $x$ is any variable appearing in $X'$ then $(I:x)/I$ has a linear resolution as an $S/(x)$-module.  Furthermore, its Betti numbers are precisely the difference between those of $I_k(X')$ and $I_k(X'')$ where $X''$ is the matrix $X'$ with $x$ substituted for zero. 
\end{cor}
\begin{proof}
Let $F_\bullet$ be a minimal free resolution of $S/I$.  By inductively applying Theorem \ref{EN proof}, we see that the minimal free resolution of $S/I_k(X'')$ can be obtained by pruning $F_\bullet$.   This precisely says that the minimal free resolution of $S/I_k(X'')$ is a direct summand of $F_\bullet \otimes S/(x)$.  Since $H_1(F_\bullet \otimes S/(x)) = \Tor^S_1(S/I,S/(x))\cong (I:x)/I$, Proposition \ref{gen characterization of pruning} shows that the resolution of $(I:x)/I$ injects into $(F_\bullet\otimes S/(x))[1]$ and hence has a linear resolution.    The statement about Betti numbers follows since from the short exact sequence of complexes used in the proof of Proposition \ref{gen characterization of pruning}
\end{proof}

\section{Applications and Examples}
Merle and Giusti's result in \cite{MR708329} was particularly beautiful because it showed that several invariants of $I_k(X')$ depended only on one number - the length of the perimeter of the largest subrectangle of zeros in the sparse generic matrix $X'$.  In this vein, Corollary \ref{proj dim and reg are constant} can be interpreted as saying that the projective dimension depends only on the number of columns that are identically zero.  The next natural question seems to be how the Betti numbers depend on the placement of zeros in the matrix.  Notice that if $\codim I_k(X') = n-k+1$ then $I_k(X')$ is a perfect ideal, and hence the Eagon-Northcott complex itself is a resolution. 

In smaller codimension, however, it is easy to produce matrices with the same perimeter of zeros, but yet whose ideals have a different number of minimal generators.   One might hope that the perimeter and number of generators are sufficient to compute all the Betti numbers.  However, the following example shows two matrices that give rise to ideals with the same codimension and number of generators, but have different Betti numbers.

\begin{ex}
$$\begin{array}{c|c|c|c}
X'  & \codim I_3(X')& \begin{tabular}{c} perimeter \\ of zeros \end{tabular} & \mbox{Betti Table of }(S/I_3(X'))\\ & & &\\
\begin{pmatrix}  \textbf{0}&  \textbf{0}&  \textbf{0}& x & y & z \\ 
\textbf{0} & \textbf{0} & \textbf{0} & a & b &c\\
d & e & f & g &h & w \end{pmatrix} \ & 2& 10 & \begin{array}{|ccccc}\hline 1 & - & - & - & - \\
 - & - & - & - & - \\
 - & 10 & 18 & 12 & 3
 \end{array} \\ & &&
\\
\begin{pmatrix}
 \textbf{0}& \textbf{0} & \textbf{0} & \textbf{0} & a & b \\
 \textbf{0}& \textbf{0} & c & d & \textbf{0} & \textbf{0}\\
e & f & g & h & \textbf{0} & \textbf{0}
\end{pmatrix} &  2& 10 & \begin{array}{|ccccc}\hline 1 & - & - & - & - \\
 - & - & - & - & - \\
 - & 10 & 17 & 10 & 2 
 \end{array}
\end{array}
$$
\end{ex}

This suggests that whatever dependence the Betti numbers have on the arrangement of zeros is subtle.   However, in the case of codimension $n-k$, we have the following:

\begin{thm}
Let $X'$ be a sparse $k\times n$ generic matrix and let $I'=I_k(X')$. If $\codim I' = n-k$  then the Betti numbers of $S/I'$ depend only on the number of identically vanishing minors of $X'$. 
\end{thm}

The proof follows from the more general lemma from Boij-Soderberg Theory (\cite{Boij:2008fk}):

\begin{lemma}
If $I$ is an ideal generated in degree $d$ with a linear resolution such that $\codim I = \projdim S/I - 1$ then the Betti table of $S/I$ is determined by the minimal number of generators $\mu(I)$.
\end{lemma}

\begin{proof}
Let $\projdim S/I = r$.  By Boij-Soderberg Theory, the Betti table of $S/I$ is a linear combination over $\QQ$ of two pure diagrams $B_1$ and $B_2$ corresponding to the sequences 
$$(0,d,d+1,\ldots, d+r), \ \mbox{and} \ (0,d,d+1,\ldots,d+r-1)$$
respectively.  If $\beta(S/I)$ denotes the Betti table of $S/I$ then we have 
$$\beta(S/I) = a_1 B_1 + a_2 B_2.$$
By equating the zeroth and first Betti numbers on each side, we obtain the following equations
\begin{eqnarray*}
a_1 +a_2 &=&1 \\  
{d+r \choose d} a_1+ {d+r-1 \choose d} a_2 &= &\mu(I)\end{eqnarray*}
from which we can determine $a_1,a_2$ and hence $\beta(S/I)$.  
\end{proof}



Next, we answer a question of Giusti and Merle concerning when the ideals $I_k(X')$ are radical.

\begin{prop}
If $X'$ is any sparse generic matrix, then the nonzero minors are a universal Gr\"obner basis for the ideal they generate.  In particular, for each term order, the initial ideal is squarefree, and thus $I_k(X')$ is a radical ideal.
\end{prop}

\begin{proof}
Let $Y$ be a generic $k\times n$ matrix with entries $z_{ij}$ and $x_{ij}$ corresponding to the zero and nonzero entries of $X'$ respectively.  Let $<$ be any term order on the variables supporting $I_k(X')$.  Then extend this to an order $<_2$ on the $z_{ij}$ where the $z_{ij}$ are weighted last.  Let $f \in I_k(Y)$.  Then if $f = \sum c_J \Delta_J(X')$ is nonzero, consider the element
$$\overline{f} = \sum c_J \Delta_J (Y).$$
Then since the $z_{ij}$ are weighted last, $\initial \overline{f}=\initial f$.  And thus $\initial f$ is divisible by some $m_0 =\initial \Delta_J(Y)= \initial \Delta_J(X')$.  
\end{proof}

\subsection{Monomial Ideals with Linear Resolutions}
A corollary of our work is that we can produce many monomial ideals in any degree that have linear resolutions.   For example, by Theorem \ref{all initial ideals have linear resolutions}, we know that if we choose any monomial term order $<$ and any generic matrix $X$, then the initial ideal $I_k(X)$ with respect to $<$ has a linear resolution.  The proof of this fact carries through to work for generic matrices with zeros as well.   Also, in the spirit of the proof of Theorem \ref{all initial ideals have linear resolutions} we can also set any entries in the same column equal to each other, and obtain yet another ideal with a linear resolution.   Hence we have the following:

\begin{thm}
Let $X'$ be a generic $k\times n$ matrix with zeros and let $<$ be any monomial term order.  Then the initial ideal $J = \initial_< I_k(X')$ is an ideal with a linear resolution.  Furthermore, if $\{(x_i,y_i)\}$ is any collection of variables such that for each $i$, $x_i$ and $y_i$ are in the same column of $X'$ then the ideal $J_{x\to y}$ where we substitute $y_i$ for $x_i$ still has a linear resolution.
\end{thm}

If we apply this theorem by setting each variable in each column to the same variable (say $y_i$) then we will obtain a squarefree monomial ideal in $K[y_1,\ldots,y_n]$ which has a linear resolution.  This proves

\begin{cor}
Let $X'$ be a generic $k\times n$ matrix with zeros.  Let $J$ denote the ideal generated by all such $\prod y_{i_1}\cdots y_{i_k}$ such that the $\det X'_{i_1,\ldots,i_k}\neq 0$.  Then $J$ has a linear resolution.  
\end{cor}

\subsection{Questions and Future Work}
It is interesting to ask to what extent the pruning technique works in general.  There are two directions in which one could attempt to answer this question:

\begin{question}~
\begin{enumerate}\label{questions}
\item For what other classes of ideals does the pruning technique compute a minimal resolution after setting variables equal to zero?  For example, what can be said for determinantal ideals of lower order minors of sparse generic matrices.
\item How does pruning work when we prune by setting arbitrary linear forms equal to zero?  For example, when can we use a pruning technique to compute the minimal free resolution of determinantal ideals of (non-generic) matrices of linear forms?
\end{enumerate}
\end{question}

One interesting case for Question $(ii)$ is the resolution of the ideal of $2\times 2$ minors of an arbitrary $2\times n$ matrix of linear forms.   In \cite{MR1772515}, the authors computed Gr\"obner bases and a free resolution of all such ideals.  In the cases where the matrix is sparse generic, our resolution agrees with theirs, but they show that in general the regularity can be as large as $n-1$.  It is not clear how a pruning technique could be used to prune the linear Eagon-Northcott complex to a nonlinear resolution.  However, there may be an interpretation via mapping cones as in Remark \ref{rmk about mapping cones}.

Another special case of Question \ref{questions} is the case when the linear forms are the difference of two variables.  In other words, how does the minimal free resolution of an ideal change as variables are set equal to one another?  This question must necessarily be difficult, since any ideal can be obtained from a generic complete intersection (in many variables) by successively setting variables equal to one another.  However, in some cases it may be possible to give an effective answer. 





\section*{Acknowledgments} 
Throughout the course of this project, many calculations were performed using the software Macaulay2 \cite{M2}.  The author is grateful to David Eisenbud, Daniel Erman, Andy Kustin,  and Claudiu Raicu for many helpful conversations.  Finally, we thank the reviewer for several helpful suggestions. 

 \bibliographystyle{mrl}

\end{document}